\documentclass{amsart}
\usepackage{graphicx}
\usepackage{hyperref}
\usepackage{amsmath,amssymb,amsfonts,amscd,amsthm}
\usepackage{MnSymbol}
\usepackage{tikz-cd}
\usepackage{pgfplots}
\pgfplotsset{compat=1.18}
\numberwithin{equation}{section}
\numberwithin{figure}{section}

\DeclareMathOperator{\Gal}{Gal}

\DeclareMathOperator{\Aut}{Aut}
\DeclareMathOperator{\PrePer}{PrePer}

\theoremstyle{plain}
\newtheorem{lem}{\protect\lemmaname}[section]
\theoremstyle{plain}
\newtheorem{thm}{\protect\theoremname}[section]
\theoremstyle{plain}
\newtheorem{prop}{\protect\propname}[section]
\theoremstyle{plain}
\newtheorem{conj}{\protect\conjname}[section]
\theoremstyle{plain}
\newtheorem{cor}{\protect\corname}[section]
\theoremstyle{definition}
\newtheorem{defn}{\protect\defname}[section]
\theoremstyle{definition}
\newtheorem{exam}{\protect\examname}[section]
\theoremstyle{remark}

\providecommand{\lemmaname}{Lemma}
\providecommand{\remarkname}{Remark}
\providecommand{\theoremname}{Theorem}
\providecommand{\propname}{Proposition}
\providecommand{\defname}{Definition}
\providecommand{\conjname}{Conjecture}
\providecommand{\corname}{Corollary}
\providecommand{\examname}{Example}

\title{Arboreal Galois groups of rational maps with nonreal Julia sets}
\author{Chifan Leung}
\date{November 2024}
\address{Chifan Leung, Beijing International Center for Mathematical Research, Peking University, Yiheyuan Road, Beijing China, 100871}

\begin{document}
\begin{abstract}
We prove a non-abelian arboreal Galois group result for certain maps with non-real Julia set at an archimedean place. We investigate the question of determining which polynomials defined over $\mathbb{R}$ have real Julia set. Finally we show that for some certain classes of Lattès maps associated to the duplication map on an elliptic curve has non-abelian arboreal Galois groups.
\end{abstract}
\maketitle

\section{Introduction}
Let $K$ be a number field with algebraic closure $\bar{K}.$ Let $f\in K(X)$ be a rational map of degree $d\geq2,$ look at this $f$ as a dynamical system $f:\mathbb{P}^{1}(\bar{K})\rightarrow\mathbb{P}^{1}(\bar{K}).$ Denote $f^{n}=f\circ f\circ\cdots\circ f$ the map $f$ composed with itself $n$ times.
 
Fix a point $\alpha\in\mathbb{P}^{1}(K)$ and assume $\alpha$ is not an exceptional point of the map $f,$ this means that the backward orbit $f^{-\infty}(\alpha)=\bigcup_{n\geq0}f^{-n}(\alpha)$ of $\alpha$ is an infinite set. For each $n\geq1,$ define the $n$-th inverse image of the pair $(f,\alpha)$ by $$f^{-n}(\alpha)=\{\beta\in\mathbb{P}^{1}(\bar{K}):f^{n}(\beta)=\alpha\}.$$ For each $n\geq1,$ let $K_{n}=K_{n}(f,\alpha)$ be the field generated over $K$ by $f^{-n}(\alpha).$

As $f$ is defined over $K,$ and generators of $K_{n}$ are $f$-images of elements in $K_{n+1},$ it follows that $K=K_{0}\leq K_{1}\leq\cdots$ is a tower of Galois extensions of $K,$ and we define $$K_{\infty}=K_{\infty}(f,\alpha)=\bigcup_{n\geq1}K_{n}(f,\alpha).$$

The Galois group of the extension $K_{\infty}/K$ acts faithfully on an infinite rooted tree $T_{\infty}$ whose vertices are indexed by points in $f^{-\infty}(\alpha).$

\begin{defn}
The resulting injective homomorphism $$\rho:\Gal(K_{\infty}/K)\hookrightarrow\Aut(T_{\infty})$$ is known as an \emph{arobreal Galois representation} associated to $(f,\alpha).$
\end{defn}

The first mathematician to consider $\rho:\Gal(K_{\infty}/K)\rightarrow\Aut(T_{\infty})$ was Odoni, motivated by a question in elementary number theory on prime divisors of recursively defined sequences. There is an analogy with $p$-adic Galois representations associated to elliptic curves and abelian varieties as well.

Much of the early work in the study of arboreal Galois representations was about identifying examples of large image, namely to prove $\rho:\Gal(K_{\infty}/K)\rightarrow\Aut(T_{\infty})$ is surjective, or prove that $\rho(\Gal(K_{\infty}/K))$ is a finite index subgroup in $\Aut(T_{\infty}).$ See Odoni \cite{MR805714}, \cite{MR813379}, Stoll \cite{MR1174401} and Jones \cite{MR3220023} for more information.

We are interested not in large image results, but rather small image results in a sense of the following conjecture by Andrews-Petsche. They conjectured that the arboreal Galois group is abelian if and only if the polynomial is conjugated to either a powering map or to a Chebyshev map.

\begin{defn}
Let $L/K$ be extension of fields. Let $f,g\in K[X]$ and let $\alpha,\beta\in K,$ we say two pairs $(f,\alpha)$ and $(g,\beta)$ are \emph{$L$-conjugate} if $g=\varphi\circ f\circ\varphi^{-1}$ and $\beta=\varphi(\alpha)$ for some affine automorphism $\varphi=aX+b\in L[X]$ with $a\ne0.$
\end{defn}

\begin{conj} [Andrews-Petsche \cite{arxiv2001.00659}] \label{Andrews Petsche}
Let $f\in K[X]$ be a polynomial of degree $d\geq2$ defined over a number field $K,$ and let $\alpha\in K$ be a non-exceptional point for $f.$ Then $K_{\infty}/K$ is an abelian extension if and only if $(f,\alpha)$ is $K^{\mathrm{ab}}$-conjugate to $(g,\beta)$ occurring in one of the two families of examples:
\begin{enumerate}
\item $(g,\beta)=(X^{d},\zeta)$ where $\zeta$ is a root of unity in $\bar{K}.$
\item $(g,\beta)=(\pm T_{d},\zeta+\zeta^{-1})$ where $T_{d}$ is the degree $d$ Chebyshev polynomial.
\end{enumerate}
\end{conj}

Conjecture \ref{Andrews Petsche} remains open, but some partial results have been established. We recall $f$ is \emph{postcritically finite}, or \emph{PCF} for short, if all of its critical points have finite forward orbit, or equivalently all of its critical points are preperiodic.

\begin{enumerate}
\item Conjecture \ref{Andrews Petsche} is true if $f$ is $\bar{K}$ conjugate to a powering or Chebyshev map by Andrews-Petsche \cite{arxiv2001.00659}.
\item Conjecture \ref{Andrews Petsche} is true if $K=\mathbb{Q},$ due to results of Ostafe \cite{Ostafe:intersection} and Ferraguti-Ostafe-Zannier \cite{arxiv2203.10034} in combination of Kronecker-Weber theorem.
\item Conjecture \ref{Andrews Petsche} is true if $f$ is not postcritically finite, due to Ferraguti-Ostafe-Zannier \cite{arxiv2203.10034}, that $\Gal(K_{\infty}/K)$ is not abelian if $f$ is not PCF.
\end{enumerate}

We prove the following result, which is stated in Theorem \ref{nonperiodic archimedean} for convenience, that if we can find a real archimedean place of a number field such that the Julia set is not contained in $\mathbb{P}^{1}(\mathbb{R}),$ it follows that $\Gal(K_{\infty}/K)$ is not abelian.

\begin{thm} \label{main theorem}
Let $K$ be a number field, and $v$ a real archimedean place of $K$ corresponding to a map $\sigma:K\hookrightarrow\mathbb{R}.$ Let $f\in K(X)$ be a rational map of degree $\geq2,$ and let $\alpha\in\mathbb{P}^{1}(K)$ be a nonperiodic point, and assume the following.
\begin{enumerate}
\item The rational map $f_{\sigma}:\mathbb{P}^{1}(\mathbb{R})\rightarrow\mathbb{P}^{1}(\mathbb{R})$ is surjective, where $f_{\sigma}$ is the polynomial in $\mathbb{R}[X]$ defined by applying $\sigma$ to the coefficients of $f.$
\item The Julia set of $f_{\sigma}$ is not contained in $\mathbb{P}^{1}(\mathbb{R}).$
\end{enumerate}
Then $\Gal(K_{\infty}/K)$ is not abelian.
\end{thm}

The proof of Theorem \ref{main theorem} uses the fact that if $\beta\in\bar{K}$ is in an abelian extension of a number field $K,$ then for any place $v$ of $K,$ the completion $K_{v}$ has either all or none of $\Gal(\bar{K}/K)$-conjugates of $\beta,$ so it is impossible to have partial splitting in the case if $\Gal(K_{v}/K)$ is abelian. We will prove this in Lemma \ref{all root in it}.

With this tool, we can characterize certain classes of polynomials such that the arboreal Galois group is not abelian.

We begin by stating the definition of critical interval of a polynomial $f\in\mathbb{R}[X].$ We will see in Theorem \ref{Julia equivalence} that any nonempty set of real preperiodic points $\PrePer(f,\mathbb{R})$ of a polynomial $f\in\mathbb{R}[X]$ if and only if its Julia set is real.

\begin{defn}
If $f\in\mathbb{R}[X]$ is a degree $d\geq2$ polynomial with real coefficients, denote $\mathrm{I}(f)$ the \emph{critical interval,} it is the set of all $t\in\mathbb{R}$ such that $f(X)=t$ has $d$ real solutions counted with multiplicity.
\end{defn}

In particular we can prove these propositions, we state again for convenience. We will prove these as Propositions \ref{odd positive}, \ref{odd negative}, \ref{even positive}, \ref{even negative} respectively.

\begin{prop}
Let $f\in\mathbb{R}[X]$ have odd degree $d\geq3$ with positive lead coefficient. Then $\mathrm{J}(f)\subseteq\mathbb{R}$ if and only if $\mathrm{Fix}(f,\mathbb{R})\subseteq\mathrm{I}(f).$
\end{prop}

\begin{prop}
Let $f\in\mathbb{R}[X]$ have odd degree $d\geq3$ with negative lead coefficient. Then $\mathrm{J}(f)\subseteq\mathbb{R}$ if and only if $\mathrm{Fix}(f^{2},\mathbb{R})\subseteq\mathrm{I}(f^{2}).$
\end{prop}

\begin{prop}
Let $f\in\mathbb{R}[X]$ have even degree $d\geq2$ with positive lead coefficient. Then $\mathrm{J}(f)\subseteq\mathbb{R}$ if and only if $f$ has at least one real fixed point and $[a_{1},a_{2}]\subseteq\mathrm{I}(f),$ where $a_{2}$ is the largest real fixed point, and $a_{1}$ is the smallest real root of $f(X)=a_{2}.$
\end{prop}

\begin{prop}
Let $f\in\mathbb{R}[X]$ have even degree $d\geq2$ of negative lead coefficient. Then $\mathrm{J}(f)\subseteq\mathbb{R}$ if and only if $f$ has at least one real fixed point and $[a_{1},a_{2}]\subseteq\mathrm{I}(f),$ where $a_{1}$ is the smallest real fixed point, and $a_{2}$ is the largest real root of $f(X)=a_{1}.$
\end{prop}

In Section $6$ we show that for some certain classes of Lattès maps associated to the duplication map on an elliptic curve $E,$ if the Weierstrass equation of $y^{2}=F(x)$ has negative discriminant, then any nonperiodic point of a Lattès map has nonabelian arboreal Galois group. First we recall the definition of a Lattès map.

\begin{defn}
We say $f:\mathbb{P}^{1}\rightarrow\mathbb{P}^{1}$ is a \emph{Lattès map} if there is an elliptic curve $E,$ a morphism $\psi:E\rightarrow E$ and a map $\pi: E\rightarrow\mathbb{P}^{1}$ such that the diagram commmutes \begin{equation}
\begin{tikzcd}
E\arrow[r,"\psi"]\arrow[d,"\pi"]&E\arrow[d,"\pi"]\\\mathbb{P}^{1}\arrow[r,"f"]&\mathbb{P}^{1}.
\end{tikzcd}
\end{equation}
\end{defn}

\begin{cor}
Let $K$ be a number field, and $E$ an elliptic curve defined over $K.$ Let $y^{2}=F(x)$ be a Weierstrass equation for $E,$ where $F\in\mathbb{R}[x]$ has degree three. Suppose that $\mathrm{disc}(F)<0.$  If $f:\mathbb{P}^{1}\rightarrow\mathbb{P}^{1}$ is a Lattès map associated to $[2]:E\rightarrow E$ and $x:E\rightarrow\mathbb{P}^{1}$ such that $x\circ[2]=f\circ x,$ and if $\alpha\in\mathbb{R}$ is a nonperiodic point of $f,$ then $\Gal(K_{\infty}/K)$ is not abelian.
\end{cor}

\section{Results on equidistribution of dynamically small points}
We briefly review the notion of the canonical height associated to a rational map, and the notion of local equidistribution of dynamically small points.

\begin{defn}
If $f$ is a degree $d\geq2$ rational map defined over a number field $K,$ the \emph{Call Silverman canonical height} for $f$ is the map $\hat{h}_{f}:\mathbb{P}^{1}(\bar{K})\rightarrow\mathbb{R}$ given by $$\hat{h}_{f}(x)=\lim_{n\rightarrow\infty}\frac{h(f^{n}(x))}{d^{n}},$$
where $h$ is the Weil height.
\end{defn}

\begin{prop} \label{canonical height}
The canonical height $\hat{h}_{f}$ is characterized by the two properties.
\begin{enumerate}
\item For all $x\in\mathbb{P}^{1}(\bar{K})$ one has $\hat{h}_{f}(x)=h(x)+O(1).$
\item For all $x\in\mathbb{P}^{1}(\bar{K})$ one has $\hat{h}_{f}(f(x))=d\hat{h}_{f}(x).$
\end{enumerate}
The implied constant $O(1)$ depends on $f$ but it is independent of $x.$

Moreover $\hat{h}_{f}(x)\geq0$ for all $x\in\mathbb{P}^{1}(\bar{K}),$ equality if and only if $x$ is $f$-preperiodic.
\end{prop}

\begin{proof}
This is proven in Silverman \cite{MR2316407} Theorem $3.11.$ The existence of this limit uses property $(1)$ along with a telescoping series argument due to Tate.
\end{proof}

For any place $v$ of a number field $K,$ let $\mathbb{P}^{1}_{\mathrm{Berk},v}$ be the Berkovich projective line. See Benedetto \cite{MR3890051} for more background.

\begin{itemize}
\item If $v$ is archimedean, then $\mathbb{C}_{v}=\mathbb{C}$ and $\mathbb{P}^{1}_{\mathrm{Berk},v}=\mathbb{P}^{1}(\mathbb{C})$ is the Riemann sphere.
\item If $v$ is non-archimedean, then $\mathbb{P}^{1}_{\mathrm{Berk},v}$ is a strictly larger analytic compactification of the ordinary projective line $\mathbb{P}^{1}(\mathbb{C}_{v}).$
\end{itemize}

\begin{defn}
Let $K$ be a number field and $v$ be a place of $K.$ For any $\alpha\in\mathbb{P}^{1}(\bar{K}),$ denote by $[\alpha]_{v}$ the Borel probability measure on $\mathbb{P}^{1}_{\mathrm{Berk},v},$ supported equally on each of the $[K(\alpha): K]$ distinct embeddings of $\alpha$ into $\mathbb{P}^{1}(\mathbb{C}_{v})$ induced by $K(\alpha)\hookrightarrow\mathbb{C}_{v}.$ Explicitly $[\alpha]_{v}$ is the probability measure $$[\alpha]_{v}=\frac{1}{[K(\alpha): K]}\sum_{\sigma: K(\alpha)\hookrightarrow\mathbb{C}_{v}}\delta_{\sigma(\alpha)},$$ where for each $t\in\mathbb{P}^{1}_{\mathrm{Berk},v},$ define $\delta_{t}$ as the Dirac measure on $\mathbb{P}^{1}_{\mathrm{Berk},v}$ supported at $t.$ In particular if $\varphi:\mathbb{P}^{1}_{\mathrm{Berk},v}\rightarrow\mathbb{R}$ is continuous it follows that $$\int\varphi d[\alpha]_{v}=\frac{1}{[K(\alpha): K]}\sum_{\sigma: K(\alpha)\hookrightarrow\mathbb{C}_{v}}\varphi(\sigma(\alpha)).$$
\end{defn}

The following is the arithmetic equidistribution theorem, proven independently by Baker-Rumely \cite{arxiv0407426}, Chambert-Loir \cite{arxiv0304023}, and Favre-Rivera-Letelier \cite{arxiv0407471}, says that for any rational map $f\in K(x)$ of degree at least $2,$ and any place $v,$ there exists a unit Borel measure $\mu_{f,v}$ on $\mathbb{P}^{1}_{\mathrm{Berk},v}$ which describes the limiting distribution on $\Gal(\bar{K}/K)$-orbits of dynamically small points of $f.$

\begin{thm} [Baker-Rumely \cite{arxiv0407426}, Chambert-Loir \cite{arxiv0304023}, Favre-Rivera-Letelier \cite{arxiv0407471}] \label{non-archimedean equidistribution}
Let $v$ be a place of a number field $K,$ and $f\in K(X)$ a rational map degree $d\geq2.$ There exists a Borel probability measure $\mu_{f,v}$ on $\mathbb{P}^{1}_{\mathrm{Berk},v}$ with the following property: If $\{\alpha_{n}\}$ is an infinite sequence of distinct points in $\mathbb{P}^{1}(\bar{K})$ such that $\hat{h}_{f}(\alpha_{n})\rightarrow0,$ then the sequence $[\alpha_{n}]_{v}$ converges weakly to $\mu_{f,v}$ as $n\rightarrow\infty.$ In other words, $$\lim_{n\rightarrow\infty}\int\varphi d[\alpha_{n}]_{v}=\int \varphi d\mu_{f,v}$$ for each continuous map $\varphi:\mathbb{P}^{1}_{\mathrm{Berk},v}\rightarrow\mathbb{R}$ defined on the Berkovich projective line.
\end{thm}

\section{Preliminaries on dynamics of polynomials}
In this section, we give an overview of complex dynamics, such as the definitions and properties of Fatou and Julia sets. Details may be found in \cite{beardon:gtm} Chapter $3.$

\begin{defn}
Let $\mathcal{F}$ be a family of maps from metric spaces $(X,d_{X})$ to $(Y,d_{Y}).$ We say that $\mathcal{F}$ is \emph{equicontinuous} at a point $\alpha\in X$ if for all $\varepsilon>0,$ there exists $\delta>0$ such that whenever $d_{X}(\alpha,\beta)<\delta$ one has $d_{Y}(f(\alpha),f(\beta))<\varepsilon$ for all $f\in\mathcal{F}.$
\end{defn}

The following result justifies the existence of maximal open subset such that the family of iterates is equicontinuous. It follows that the Fatou set is well-defined.

\begin{thm}
If $\mathcal{F}$ is a family of maps between metric spaces $(X,d_{X})$ and $(Y,d_{Y}),$ then there is a maximal open subset $U$ in $X$ such that $\mathcal{F}$ is equicontinuous on $U.$ In particular if $f$ is a self mapping on $(X,d_{X}),$ there is a maximal open subset $U$ of $X$ on which the family of iterates $\{f^{n}:n\geq1\}$ is equicontinuous on $U.$
\end{thm}

\begin{proof}
This is proven in \cite{beardon:gtm} Theorem $3.1.2.$
\end{proof}

\begin{defn}
Let $f\in\mathbb{C}(X)$ have degree $d\geq2,$ its \emph{Fatou set} $\mathrm{F}(f)$ is the maximal open subset in $\mathbb{P}^{1}(\mathbb{C})$ such that the family of iterates $\{f^{n}:n\geq1\}$ is equicontinuous. The \emph{Julia set} $\mathrm{J}(f)$ is the complement of $\mathrm{F}(f)$ in $\mathbb{P}^{1}(\mathbb{C}).$
\end{defn}

\begin{prop}
Fatou set is open in $\mathbb{P}^{1}(\mathbb{C})$ and Julia set is compact in $\mathbb{P}^{1}(\mathbb{C}).$
\end{prop}

\begin{proof}
The Fatou set $\mathrm{F}(f)$ is open in $\mathbb{P}^{1}(\mathbb{C})$ by definition. As $\mathrm{J}(f)$ is complement of an open set, it is closed in $\mathbb{P}^{1}(\mathbb{C}),$ and compactness follows from the fact that $\mathbb{P}^{1}(\mathbb{C})$ is compact.
\end{proof}

\begin{defn}
If $f$ is a polynomial in $\mathbb{C}[X],$ the \emph{filled Julia set} of $f$ is the set of all $z\in\mathbb{C}$ such that $f^{n}(z)$ is bounded as $n\rightarrow\infty,$ it is denoted by $\mathrm{FJ}(f).$
\end{defn}

The following proposition shows that the Julia set of a polynomial $f$ is in fact the boundary of the filled Julia set.

\begin{prop} \label{Julia boundary}
The Julia set $\mathrm{J}(f)$ is the boundary of the filled Julia set $\mathrm{FJ}(f).$
\end{prop}

\begin{proof}
This is proven in \cite{Filled-Julia} Proposition $2.6.$
\end{proof}

\begin{prop} \label{filled Julia compactness}
The filled Julia set $\mathrm{FJ}(f)$ is compact in $\mathbb{P}^{1}(\mathbb{C})$
\end{prop}

\begin{proof}
We have $\mathrm{FJ}(f)$ is the complement of $U_{\infty}(f)$ in $\mathbb{P}^{1}(\mathbb{C}),$ where $U_{\infty}(f)$ is the basin of $\infty$ as a superattracting fixed point in Beardon's terminology \cite{beardon:gtm} page $104.$ As explained by Beardon \cite{beardon:gtm}, it follows that $U_{\infty}(f)=\bigcup_{n\geq1}f^{-n}(B)$ where $B$ is the connected component of the Fatou set containing the point $\infty,$ Hence both $B$ and $U_{\infty}(f)$ are open subsets of $\mathbb{P}^{1}(\mathbb{C}).$ Thus $\mathrm{FJ}(f)$ is closed in $\mathbb{P}^{1}(\mathbb{C}),$ so it is compact as $\mathbb{P}^{1}(\mathbb{C})$ is compact.
\end{proof}

\begin{lem} \label{compact subset}
Let $X$ be a compact metric space and let $\mu_{n}$ be a sequence of Borel probability measures on $X$ converges weakly to a Borel probability measure $\mu$ on $X.$ Suppose $C$ is a compact subset in $X$ such that the support $\mathrm{supp}(\mu_{n})\subseteq C$ for all $n,$ then we have that $\mathrm{supp}(\mu)\subseteq C.$
\end{lem}

\begin{proof}
Suppose that the support $\mathrm{supp}(\mu)\nsubseteq C,$ pick $x\in\mathrm{supp}(\mu)$ that is not in $C.$ Since $x$ is in $\mathrm{supp}(\mu),$ thus every open neighborhood $U$ of $x$ has positive $\mu$-measure. By the regular property of metric spaces, there is a neighborhood $U$ of $x$ such that the closure $\bar{U}$ is disjoint from $C.$ There is a continuous map $\varphi: X\rightarrow[0,1]$ such that $\varphi(\bar{U})=1$ and $\varphi(C)=0$ by Urysohn's Lemma. We observe that $\int_{X}\varphi d\mu_{n}=0$ follows from the assumption that each $\mu_{n}$ is supported on $C,$ it follows that $$\int_{X}\varphi d\mu\geq\int_{U}\varphi d\mu=\int_{U}1d\mu=\mu(U)>0$$ which contradicts the weak convergence.
\end{proof}

\begin{thm}[Lyubich \cite{Complex-equi-distribution}] \label{archimedean version}
Let $f\in\mathbb{C}(X)$ be a degree $d\geq2$ rational function. Assume $\alpha\in\mathbb{P}^{1}(\mathbb{C})$ is not an exceptional point of $f.$ Let $\mu_{n}$ be the probability measure on $\mathbb{P}^{1}(\mathbb{C})$ supported equally on each of the $d^{n}$ points in $f^{-n}(\alpha).$ That is, we define $$\mu_{n}=\frac{1}{d^{n}}\sum_{z\in f^{-n}(\alpha)}\delta_{z},$$ where each point in the sum is counted with multiplicity and $\delta_{z}$ is Dirac measure on $\mathbb{P}^{1}(\mathbb{C})$ giving mass $1$ to the point $z.$ Then the sequence $\mu_{n}$ converges weakly to the \emph{canonical measure} $\mu_{f,\infty}$ supported on the Julia set of $f.$
\end{thm}

\begin{proof}
This is proven in Lyubich \cite{Complex-equi-distribution} Theorem $1.$
\end{proof}

The main idea behind the proof of Theorem \ref{nonperiodic archimedean} is in the following simple Lemma, as well as an  application of Theorem \ref{non-archimedean equidistribution} at an archimedean place.

\begin{lem} \label{all root in it}
Let $K$ be a number field. Assume that $F(X)$ is irreducible in $K[X],$ and let $L/K$ be the splitting field of $F.$ Suppose that $\Gal(L/K)$ is abelian, then for any place $v$ of $K,$ either $F$ splits completely over $K_{v},$ or it has no roots in $K_{v}.$
\end{lem}

\begin{proof}
If $\alpha$ is a root in $K_{v},$ the abelian assumption implies $K(\alpha)$ is Galois over $K.$ Since $F(X)$ is irreducible in $K[X],$ it follows that $F$ splits completely over $K(\alpha)$ and hence over $K_{v},$ because $K(\alpha)$ is a subfield of $K_{v}.$
\end{proof}

\begin{thm}\label{nonperiodic archimedean}
Let $K$ be a number field, and $v$ a real archimedean place of $K$ corresponding to a map $\sigma:K\hookrightarrow\mathbb{R}.$ Let $f\in K(X)$ be a rational map of degree $\geq2,$ and let $\alpha\in\mathbb{P}^{1}(K)$ be a nonperiodic point, and assume the following.
\begin{enumerate}
\item The rational map $f_{\sigma}:\mathbb{P}^{1}(\mathbb{R})\rightarrow\mathbb{P}^{1}(\mathbb{R})$ is surjective, where $f_{\sigma}$ is the polynomial in $\mathbb{R}[X]$ defined by applying $\sigma$ to the coefficients of $f.$
\item The Julia set of $f_{\sigma}$ is not contained in $\mathbb{P}^{1}(\mathbb{R}).$
\end{enumerate}
Then $\Gal(K_{\infty}/K)$ is not abelian.
\end{thm}

\begin{proof}
If $\alpha$ is not a periodic point of $f,$ by the surjectivity of $f:\mathbb{P}^{1}(\mathbb{R})\rightarrow\mathbb{P}^{1}(\mathbb{R}),$ we may define $\alpha_{0}=\alpha,\alpha_{1},\alpha_{2},\dots$ in $\mathbb{P}^{1}(\mathbb{R})$ such that $f(\alpha_{n+1})=\alpha_{n}$ for all $n\geq1.$ These $\alpha_{n}$ are distinct since $\alpha$ is not periodic, otherwise $\alpha_{m}=\alpha_{n}$ implies that $f^{m-n}(\alpha_{m})=\alpha_{n}$ for any $m\geq n.$

Suppose that $\Gal(K_{\infty}/K)$ is abelian, these $\alpha_{n}$ are totally real by Lemma \ref{all root in it}. Thus the measures $[\alpha_{n}]_{v}$ on $\mathbb{P}^{1}(\mathbb{R})$ converges weakly to the canonical measure $\mu_{f,v}$ such that the support $\mathrm{supp}(\mu_{f,v})=\mathrm{J}(f)$ by Theorem \ref{non-archimedean equidistribution}. This contradicts the assumption that $\mathrm{J}(f)$ is not contained in $\mathbb{P}^{1}(\mathbb{R})$ by Lemma \ref{compact subset}.
\end{proof}

Note that the hypotheses in Theorem \ref{nonperiodic archimedean} is automatically satisfied if $K$ is a number field with a real place, $f$ is a polynomial in $K[X]$ with odd degree and nonreal Julia set, and $\alpha$ is not $f$-periodic.

\section{Characterize polynomials with real Julia sets}
\begin{thm} \label{Julia equivalence}
Assume $f\in\mathbb{R}[X]$ is a degree $d\geq2$ polynomial with real coefficients. The following are equivalent.
\begin{enumerate}
\item We have that $\mathrm{J}(f)\subseteq\mathbb{R},$\ in other words the Julia set is real.
\item We have that $\mathrm{FJ}(f)\subseteq\mathbb{R},$\ in other words the filled Julia set is real.
\item The set of complex preperiodic points $\PrePer(f,\mathbb{C})\subseteq\mathbb{R}.$
\item The set of real preperiodic points $\PrePer(f,\mathbb{R})$ is nonempty and is contained in the critical interval $\mathrm{I}(f).$
\end{enumerate}
\end{thm}

\begin{proof}
First we prove the implication $(1)\Rightarrow(2).$ If $\mathrm{FJ}(f)$ is not contained in $\mathbb{R},$ then it contains a point with nonzero imaginary part. Without loss of generality, we can assume the point has positive imaginary part, as the proof for a point of negative imaginary part is similar. Because $\mathrm{FJ}(f)$ is compact by Proposition \ref{filled Julia compactness}, it contains a point $\zeta$ having maximal positive imaginary part. This $\zeta$ would be a nonreal boundary in $\mathrm{FJ}(f),$ which contradicts to the fact that $\mathrm{J}(f)\subseteq\mathbb{R}$ is the boundary of $\mathrm{FJ}(f).$ Thus $\mathrm{FJ}(f)$ is real, so it has empty interior. In particular $\mathrm{FJ}(f)=\mathrm{J}(f)$ because any closed set with empty interior is its own boundary.

Note that implication $(2)\Rightarrow(3)$ follows immediately from $\PrePer(f,\mathbb{C})\subseteq\mathrm{FJ}(f)$ because preperiodic points must have bounded forward orbit.

Next we prove $(3)\Rightarrow(4).$ If $(4)$ is false, either $\PrePer(f,\mathbb{R})$ is empty, or it is not contained in $\mathrm{I}(f).$ If $\PrePer(f,\mathbb{R})$ is empty, then $\mathrm{Fix}(f,\mathbb{C})$ has nonreal solutions, this will imply that $\PrePer(f,\mathbb{C})\nsubseteq\mathbb{R}$ as fixed points of $f$ are certainly preperiodic. Suppose that $\alpha\in\PrePer(f,\mathbb{R})\setminus\mathrm{I}(f),$ there exists a preperiodic point $\beta\in f^{-1}(\alpha)\setminus\mathbb{R}$ because $f(\beta)=\alpha$ implies $\beta$ has bounded forward orbit, so it is in $\PrePer(f,\mathbb{C})\setminus\mathbb{R}.$ In both cases we have shown that $\PrePer(f,\mathbb{C})\nsubseteq\mathbb{R}$ as required.

Lastly we show that $(4)\Rightarrow(1).$ We first prove that if $\PrePer(f,\mathbb{R})$ is nonempty, it contains at least one non-exceptional point $\alpha.$ If the only point in $\PrePer(f,\mathbb{R})$ is actually an exceptional point, then there exists a real affine automorphism $\varphi=hX+k\in\mathbb{R}[X]$ such that $\varphi\circ f\circ\varphi^{-1}=\pm X^{d}$ with $\varphi(\alpha)=0.$ Now $\pm X^{d}$ has other preperiodic points such as $\pm1.$ Thus $\varphi(0)$ and $\varphi(\pm1)$ are preperiodic, this contradicts the assumption $\PrePer(f,\mathbb{R})$ contains only one point which is in fact exceptional. Suppose that $\alpha\in\PrePer(f,\mathbb{R})\subseteq\mathrm{I}(f),$ we have $f^{-1}(\alpha)\subseteq\PrePer(f,\mathbb{C})$ for all $n\geq1,$ as an inverse image of $\alpha$ is also preperiodic. Note $f^{-1}(\alpha)\subseteq\mathbb{R}$ as $\alpha$ is an element in $I(f)$ using the definition of $I(f).$ It follows that $$f^{-1}(\alpha)\subseteq\PrePer(f,\mathbb{R})\Rightarrow f^{-1}(\alpha)\subseteq\mathrm{I}(f).$$ Iterating this argument shows that $$f^{-n}(\alpha)\subseteq\PrePer(f,\mathbb{R})\ \text{for all}\ n\geq1.$$ As $\alpha$ is not an exceptional point of $f$ and $f^{-n}(\alpha)\subseteq\PrePer(f,\mathbb{R}),$ the backward orbit $f^{-\infty}(\alpha)$ equidistributes on $\mathbb{R}$ by Theorem \ref{archimedean version}, so $\mu_{n}$ converges weakly to the canonical measure $\mu_{f,\infty}$ supported on $\mathrm{J}(f),$ and thus $\mathrm{J}(f)\subseteq\mathbb{R}$ by Lemma \ref{compact subset}.
\end{proof}

The following result shows that for any polynomial with positive lead coefficient, if we start at an initial point which is greater than largest fixed point of polynomial, it follows that forward iterates are unbounded. 

\begin{lem}\label{geometric argument}
If $f\in\mathbb{R}[X]$ is a degree $d\geq2$ polynomial with positive lead coefficient and $\alpha_{0}$ is its largest real fixed point, then $\lim_{n\rightarrow\infty}f^{n}(x)=+\infty$ for all $x>\alpha_{0}.$
\end{lem}

\begin{proof}
We claim that if $\delta>0$ is arbitrary, then $f(x)\geq \alpha_{0}+\delta$ whenever $x\geq \alpha_{0}+\delta.$ Now $\alpha_{0}=\max\{x\in\mathbb{R}:f(x)=x\}$ is the largest fixed point from the definition of $\alpha_{0}.$ Thus $f(x)-x>0$ for all $x\geq \alpha_{0}+\delta.$ In particular $f$ takes $[\alpha_{0}+\delta,\infty)$ to itself.

If $x>\alpha_{0}$ is arbitrary and $\delta$ is small enough such that $\alpha_{0}+\delta<x,$ it follows from the intermediate value theorem that $$c:=\inf_{x>\alpha_{0}+\delta}(f(x)-x)>0.$$ To prove this we let $g(x)=f(x)-x,$ then $g(x)>0$ for all $x\geq\alpha_{0}+\delta$ and $\lim_{x\rightarrow+\infty}g(x)=+\infty.$ Since $g$ is continuous, it achieves its minimum $c$ on $[\alpha_{0}+\delta,+\infty)$ and in particular $c>0.$ From the previous paragraph we know that $f(x)>x$ and $f^{2}(x)>f(x)$ from the fact that $f(x)>\alpha_{0},$ so by induction it follows that $f^{n}(x)>f^{n-1}(x)$ for all $n\geq1.$ The result follows because $$f^{n}(x)-x=(f^{n}(x)-f^{n-1}(x))+\cdots+(f(x)-x)\geq nc.$$ Each term in parentheses is greater than $c$ and thus $\lim_{n\rightarrow\infty}f^{n}(x)=\infty.$
\end{proof}

\begin{prop} \label{odd positive}
Let $f\in\mathbb{R}[X]$ have odd degree $d\geq3$ with positive lead coefficient. Then $\mathrm{J}(f)\subseteq\mathbb{R}$ if and only if $\mathrm{Fix}(f,\mathbb{R})\subseteq\mathrm{I}(f).$
\end{prop}

\begin{proof}
Suppose $\mathrm{J}(f)\subseteq\mathbb{R}$ then $\mathrm{Fix}(f,\mathbb{R})\subseteq\PrePer(f,\mathbb{R})\subseteq\mathrm{I}(f)$ by Theorem \ref{Julia equivalence}.

Suppose now that $\mathrm{Fix}(f,\mathbb{R})\subseteq\mathrm{I}(f),$ we prove that $\PrePer(f,\mathbb{R})$ is nonempty and is contained in the critical interval $\mathrm{I}(f).$ This implies $\mathrm{J}(f)\subseteq\mathbb{R}$ by Theorem \ref{Julia equivalence}. It is clear $\PrePer(f,\mathbb{R})\ne\varnothing,$ as the odd degree polynomial $f(x)-x$ has at least one real root. Also $\PrePer(f,\mathbb{R})$ contains at least one non-exceptional point as proven in Theorem \ref{Julia equivalence}. Suppose that $a_{1}$ and $a_{2}$ are the smallest and the largest real fixed points of $f$ respectively. The claim $\lim_{n\rightarrow\infty}f^{n}(x)=+\infty$ for all $x>a_{2}$ follows immediately from Lemma \ref{geometric argument}. To prove $\lim_{n\rightarrow\infty}f^{n}(x)=-\infty$ for all $x<a_{1},$ apply Lemma \ref{geometric argument} to $g(x)=-f(-x)$ so that $-a_{1}$ is the largest real fixed point of $g.$ This implies that $\lim_{n\rightarrow\infty}g^{n}(x)=+\infty$ for all $x>-a_{1}.$ It follows that $\PrePer(f,\mathbb{R})\subseteq[a_{1},a_{2}]\subseteq\mathrm{I}(f)$ and $\mathrm{J}(f)\subseteq\mathbb{R}$ by Theorem \ref{Julia equivalence}.
\end{proof}

The following result shows that a polynomial and any one of its iterates have the same Julia set.

\begin{lem} \label{critical interval iterates}
Let $f\in\mathbb{C}[X]$ have degree $d\geq2,$ then $\mathrm{I}(f^{n})\subseteq\mathrm{I}(f)$ for all $n\geq1.$
\end{lem}

\begin{proof}
That $\alpha\in\mathrm{I}(f^{n})$ if and only if $f^{n}-\alpha$ has $d^{n}$ real solutions counting multiplicity. This requires that $\beta\in\mathbb{R}$ for all $\beta\in f^{-1}(\alpha)$ because $\beta$ must be the image under $f^{n-1}\in\mathbb{R}[X]$ of one of those $d^{n}$ inverse images of $\alpha$ under $f^{n}.$ It follows that $f-\alpha$ must have $d$ real solutions and thus $\alpha$ is in the critical interval $\mathrm{I}(f).$
\end{proof}

The following result shows the relationship of the Julia set of a polynomial is invariant with that of an iterate.

\begin{lem} \label{Julia set iterates}
Let $f\in\mathbb{R}[X]$ have degree $d\geq2,$ then $\mathrm{J}(f^{n})=\mathrm{J}(f)$ for all $n\geq1.$
\end{lem}

\begin{proof}
We first show that $\mathrm{FJ}(f^{n})=\mathrm{FJ}(f),$ then $\mathrm{J}(f^{n})=\mathrm{J}(f)$ follows immediately because it is the boundary of $\mathrm{FJ}(f).$

Suppose that $x\in\mathrm{FJ}(f),$ then the $f$-orbit of $x$ is bounded, thus so is the $f^{n}$-orbit as it is a subsequence. This implies that $x\in\mathrm{FJ}(f^{n})$ which proves $\mathrm{FJ}(f)\subseteq\mathrm{FJ}(f^{n}).$

Suppose that $x\notin\mathrm{FJ}(f),$ then the $f$-orbit of $x$ converges to $\infty$ in $\mathbb{P}^{n}(\mathbb{C})$ as $\infty$ is an attracting fixed point, so do the $f^{n}$-orbit and thus $x\notin\mathrm{FJ}(f^{n}).$
\end{proof}

The following result shows the relationship of the filled Julia set of a given polynomial and any of its affine conjugates.

\begin{lem}\label{affine conjugation}
Let $f\in\mathbb{R}[X]$ be a polynomial, and $\varphi=mX+n$ an $\mathbb{R}$-affine map. Suppose that $g=\varphi\circ f\circ\varphi^{-1},$ then we have $\mathrm{FJ}(g)=\varphi(\mathrm{FJ}(f)).$
\end{lem}

\begin{proof}
Note that $\beta\in\mathrm{FJ}(g)$ if and only if $g^{n}(\beta)$ is bounded for all $n$ by definition. Let $\varphi$ be an $\mathbb{R}$ affine map such that $\varphi(\alpha)=\beta.$ Then $g^{n}(\varphi(\alpha))$ is bounded for all $n.$ Apply $\varphi^{-1}$ to get $f^{n}(\alpha)$ is bounded for all $n$ so $\alpha\in\mathrm{FJ}(f)$ and thus $\beta\in\varphi(\mathrm{FJ}(f)).$ The reverse inclusion is proven in a similar manner as $\varphi$ is bijective.
\end{proof}

\begin{prop} \label{odd negative}
Let $f\in\mathbb{R}[X]$ have odd degree $d\geq3$ with negative lead coefficient. Then $\mathrm{J}(f)\subseteq\mathbb{R}$ if and only if $\mathrm{Fix}(f^{2},\mathbb{R})\subseteq\mathrm{I}(f^{2}).$
\end{prop}

\begin{proof}
Since $d$ is odd, it follows that $f^{2}$ has odd degree with positive lead coefficient, so we apply Proposition \ref{odd positive} to $f^{2}.$ Since $\mathrm{J}(f^{2})=\mathrm{J}(f)$ from Lemma \ref{Julia set iterates} with $n=2,$ it follows that $\mathrm{J}(f^{2})=\mathrm{J}(f)\subseteq\mathbb{R}$ if and only if $\mathrm{Fix}(f^{2},\mathbb{R})\subseteq\mathrm{I}(f^{2}).$
\end{proof}

Here are some examples of Julia sets of an odd degree polynomial with negative lead coefficient.

\begin{exam}
Consider when $d=3$ the polynomial $f=-X^{3}+3X,$ the negative of the degree $3$ Chebyshev polynomial. The fixed points of $f^{2}$ are roots of $$f^{2}(X)-X=X(X^{2}-2)(X^{2}-4)(X^{4}-3X^{2}+1),$$ it follows that $$X\in\{0,\pm2,\pm\sqrt{2},(\pm1\pm\sqrt{5})/2\}.$$ By the chain rule, the critical points of $f^{2}(X)$ are the roots of $f^{\prime}(f(X))f^{\prime}(X)=0.$ This happens if and only if either $X=\pm1$ or $f(X)=\pm1.$ For the former $f^{2}(X)=\pm2,$ and the latter $f^{3}(X)=\pm2.$ It follows that the critical interval is $\mathrm{I}(f^{2})=[-2,2].$ Since all fixed points of $f^{2}$ are in $I(f^{2}),$ it follows that $\mathrm{J}(f)$ is real.
\end{exam}

\begin{exam}
Consider when $d=3$ the polynomial $f=-X^{3}+2X,$ the negative of the degree $3$ polynomial $X^{3}-2X.$ The fixed points of $f^{2}$ are roots of $$f^{2}(X)-X=X(X^{2}-1)(X^{2}-3)(X^{4}-2X^{2}+1)=X(X^{2}-1)^{3}(X^{2}-3),$$ it follows that $$X\in\{0,\pm1,\pm\sqrt{3}\}.$$ By the chain rule, the critical points of $f^{2}(X)$ are the roots of $f^{\prime}(f(X))f^{\prime}(X)=0.$ This happens if and only if either $X$ or $f(X)$ is equal to the critical point $\pm\sqrt{\frac{2}{3}}.$ For the former case $f^{2}(X)=\pm\frac{4}{3}\sqrt{\frac{2}{3}},$ and the latter case $$f^{3}(X)=\pm\frac{44}{27}\sqrt{\frac{2}{3}}<f^{2}(X),$$ it follows that the critical interval $$\mathrm{I}(f^{2})=\left[-\frac{44}{27}\sqrt{\frac{2}{3}},\frac{44}{27}\sqrt{\frac{2}{3}}\right]\ \text{is properly contained in}\ \mathrm{I}(f).$$ Since $\sqrt{3}>\frac{44}{27}\sqrt{\frac{2}{3}}$ which is not in $\mathrm{I}(f^{2}),$ it follows that $\mathrm{J}(f)$ is not contained in $\mathbb{R}.$
\end{exam}

\begin{prop} \label{even positive}
Let $f\in\mathbb{R}[X]$ have even degree $d\geq2$ with positive lead coefficient. Then $\mathrm{J}(f)\subseteq\mathbb{R}$ if and only if $f$ has at least one real fixed point and $[a_{1},a_{2}]\subseteq\mathrm{I}(f),$ where $a_{2}$ is the largest real fixed point, and $a_{1}$ is the smallest real root of $f(X)=a_{2}.$
\end{prop}

\begin{proof}
Suppose $\mathrm{J}(f)\subseteq\mathbb{R}$ then $a_{1},a_{2}\in\PrePer(f,\mathbb{R})\subseteq\mathrm{I}(f)$ by Theorem \ref{Julia equivalence}.

Suppose now that $\mathrm{Fix}(f,\mathbb{R})\subseteq\mathrm{I}(f),$ we prove that $\PrePer(f,\mathbb{R})$ is nonempty and is contained in the critical interval $\mathrm{I}(f).$ This implies $\mathrm{J}(f)\subseteq\mathbb{R}$ by Theorem \ref{Julia equivalence}. From hypothesis $f$ has at least one real fixed point, this means $\PrePer(f,\mathbb{R})\ne\varnothing.$ Also $\PrePer(f,\mathbb{R})$ has at least one non-exceptional point as proven in Theorem \ref{Julia equivalence}. The claim $\lim_{n\rightarrow\infty}f^{n}(x)=+\infty$ for all $x>a_{2}$ follows immediately from Lemma \ref{geometric argument}. To prove $\lim_{n\rightarrow\infty}f^{n}(x)=+\infty$ for all $x<a_{1},$ note that for any $x<a_{1}$ one has that $f(x)>a_{2}$ because $\lim_{x\rightarrow-\infty}f(x)=+\infty$ and  $a_{1}$ is the smallest real root of $f(X)=a_{2}.$ Thus $\PrePer(f,\mathbb{R})\subseteq[a_{1},a_{2}]\subseteq\mathrm{I}(f)$ and $\mathrm{J}(f)\subseteq\mathbb{R}$ by Theorem \ref{Julia equivalence}.
\end{proof}

The following is an application of Proposition \ref{even positive}. The boundary $c=-2$ corresponds to the degree $2$ Chebyshev map, it is the smallest real number that is in Mandelbrot set.

\begin{cor}
Let $f=X^{2}+c$ for some $c\in\mathbb{R}.$ Then $\mathrm{J}(f)\subseteq\mathbb{R}$ if and only if $c\leq-2.$
\end{cor}

\begin{proof} The two fixed points of $f$ are $\frac{1-\sqrt{1-4c}}{2}$ and $\frac{1+\sqrt{1-4c}}{2}$ by the quadratic formula. Without loss of generality, assume $c\leq\frac{1}{4},$ as if $c>\frac{1}{4},$ then $f$ has non-real fixed points, and therefore $\mathrm{J}(f)\nsubseteq\mathbb{R}$ by Theorem \ref{Julia equivalence}. Now as $f$ is an even function, the other root of $X^{2}+c=\frac{1+\sqrt{1-4c}}{2}$ is equal to $-\frac{1+\sqrt{1-4c}}{2}.$ The critical interval is $\mathrm{I}(f)=[c,\infty),$ so using Proposition \ref{even positive} we thus have 
\begin{equation}
\begin{split}
\mathrm{J}(f)\subseteq\mathbb{R}\Leftrightarrow&-\frac{1+\sqrt{1-4c}}{2}\geq c\\\Leftrightarrow&-1-2c\geq\sqrt{1-4c}\\
\Leftrightarrow&(1+2c)^{2}\geq1-4c\ \text{and}\ -1-2c\geq0\\
\Leftrightarrow&4c^{2}+8c\geq0\ \text{and}\ c\leq-1/2\\
\Leftrightarrow&4c(c+2)\geq0\ \text{and}\ c\leq-1/2\\
\Leftrightarrow&c\leq-2.
\end{split}
\end{equation}
It follows that $\mathrm{J}(f)\subseteq{\mathbb{R}}$ if and only if $c\leq-2.$
\end{proof}

The following result gives a criterion for an even degree polynomial with negative leading coefficient to have real Julia set.

\begin{prop} \label{even negative}
Let $f\in\mathbb{R}[X]$ have even degree $d\geq2$ of negative lead coefficient. Then $\mathrm{J}(f)\subseteq\mathbb{R}$ if and only if $f$ has at least one real fixed point and $[a_{1},a_{2}]\subseteq\mathrm{I}(f),$ where $a_{1}$ is the smallest real fixed point, and $a_{2}$ is the largest real root of $f(X)=a_{1}.$
\end{prop}

\begin{proof}
Let $g(X)=-f(-X)$ be obtained by conjugating $f$ with the map $X\mapsto-X,$ then $g$ has positive lead coefficient with at least one real fixed point, so we can apply Proposition \ref{even positive} to it. If $b_{1}$ and $b_{2}$ are smallest and largest real fixed points of $g,$ one has $b_{1}=-a_{1}$ and $b_{2}=-a_{2}.$ It follows that $\mathrm{J}(g)\subseteq\mathbb{R}$ if and only if $[b_{1},b_{2}]\subseteq\mathrm{I}(g).$ The result follows from Lemma \ref{affine conjugation} as $\varphi(\mathrm{I}(f))=\mathrm{I}(g)$ and $\varphi([a_{1},a_{2}])=[b_{1},b_{2}].$
\end{proof}

\section{Real positive cubic polynomials with positive lead coefficient}
In this section, we describe which real cubic polynomials with positive leading coefficient have real Julia set.

\begin{prop} \label{two class cubics}
If $f=c_{3}X^{3}+c_{2}X^{2}+c_{1}X+c_{0}$ is a cubic in $\mathbb{R}[X]$ such that $c_{3}\ne0,$ then there exists a real affine automorphism $\varphi=hX+k\in\mathbb{R}[X]$ with $h\ne0$ such that $\varphi\circ f\circ\varphi^{-1}$ is conjugate to one of $X^{3}+AX+B$ or $-X^{3}+AX+B$ for some $A,B\in\mathbb{R}.$
\end{prop}

\begin{proof}
We observe that the conjugation yields  
\begin{equation}
\begin{split}
\varphi\circ f\circ\varphi^{-1}(X)=&hc_{3}\left(\frac{X-k}{h}\right)^{3}+hc_{2}\left(\frac{X-k}{h}\right)^{2}+hc_{1}\left(\frac{X-k}{h}\right)+hc_{0}+k\\=&c_{3}h^{-2}(X-k)^{3}+c_{2}h^{-1}(X-k)^{2}+c_{1}(X-k)+hc_{0}+k\\=&c_{3}h^{-2}X^{3}+(c_{2}h^{-1}-3c_{3}h^{-2}k)X^{2}+\cdots+hc_{0}+k.
\end{split}
\end{equation}
Choose $h\in\mathbb{R}$ such that $h^{2}=\pm c_{3}$ depending on whether $c_{3}$ is positive or negative. Then with this $h$ chosen, we set $k=hc_{2}/(3c_{3})\in\mathbb{R}$ which proves the proposition.
\end{proof}

In view of Propposition \ref{two class cubics}, in order to characterize real cubic polynomials with real Julia set, it suffices to look at the two forms $X^{3}+AX+B$ and $-X^{3}+AX+B,$ we treat only the first case in this paper. The second case would be interesting to treat as well. Let $f=f_{A,B}=X^{3}+AX+B$ and define $$\mathcal{R}=\{(A,B)\in\mathbb{R}^{2}:\mathrm{J}(f_{A,B})\subseteq\mathbb{R}\}.$$ The goal is to describe $\mathcal{R}$ as explicitly as possible. We invoke Theorem \ref{Julia equivalence} that for $\mathrm{J}(f)$ to be real, it is required that $\mathrm{I}(f)$ to be nonempty, which in turn requires the existence of real critical points. For purposes of calculating $\mathcal{R},$ if we suppose that $A=-3a^{2}\leq0$ for some real number $a\geq0,$ such that the critical points of $f$ are $\pm a,$ it follows that $\mathrm{I}(f)=[B-2a^{3},B+2a^{3}].$ In particular $(A,B)\notin\mathcal{R}$ whenever $A>0.$

\begin{lem}
With notation as above, $(A,B)\in\mathcal{R}$ if and only if $(A,-B)\in\mathcal{R}.$
\end{lem}

\begin{proof}
Conjugating $f_{A,B}$ by $X\mapsto-X$ results in $f_{A,-B}.$ Lemma \ref{affine conjugation} implies that $\mathrm{J}(f_{A,-B})=-\mathrm{J}(f_{A,B})$ and fixed points are mapped to their negatives. The result now follows from Theorem \ref{Julia equivalence} because $-\mathrm{J}(f_{A,B})\subseteq\mathbb{R}$ if and only if $\mathrm{J}(f_{A,B})\subseteq\mathbb{R}.$
\end{proof}

Let $\mathcal{U}$ be the set of $(A,B)\in\mathbb{R}^{2}$ such that the polynomial $f_{A,B}=X^{3}+AX+B$ has $3$ distinct real fixed points, equivalently it is the set $$\mathcal{U}=\{(A,B)\in\mathbb{R}^{2}:\mathrm{disc}(f_{A,B}(X)-X)>0\}.$$ Thus $\mathcal{U}$ is open in $\mathbb{R}^{2}$ and for all $(A,B)\in\bar{\mathcal{U}},$ the polynomial $f_{A,B}$ has $3$ real fixed points counted with multiplicity, we also have $\mathcal{R}\subseteq\bar{\mathcal{U}}.$

\begin{lem} \label{fixed point decrease}
Suppose $f=X^{3}-3a^{2}X+B$ is a cubic polynomial in $\mathbb{R}[X]$ with $B\geq0.$ Let $A=-3a^{2}$ and suppose that $(A,B)\in\bar{\mathcal{U}},$ which ensures all fixed points are real. Let $\alpha_{1}(B)$ and $\alpha_{2}(B)$ be smallest and largest fixed points of $f$ as functions on $B.$ Then both $\alpha_{1}$ and $\alpha_{2}$ decrease as $B$ increases. 
\end{lem}

\begin{proof}
We will show that $\alpha_{i}(B)$ is decreasing in $B$ for $i=1,2.$ Since $\alpha_{i}(B)$ is a fixed point, it satisfies $\alpha_{i}^{3}(B)-(3a^{2}+1)\alpha_{i}(B)+B=0.$ Now $\alpha_{i}(B)$ are differentiable as functions on $B$ for all $(A,B)\in\mathcal{U},$ because real roots of a monic polynomial with distinct real roots are differentiable functions of the coefficients, as an application of the implicit function theorem. Differentiate with respect to $B$ so that $$\alpha_{i}^{\prime}(B)(3\alpha_{i}^{2}(B)-3a^{2}-1)=-1.$$ Since $f^{\prime}(\alpha_{i}(B))=3\alpha_{i}^{2}(B)-3a^{2}>1$ by Lemma \ref{geometric argument}, it follows that $\alpha_{i}^{\prime}(B)<0.$
\end{proof}

\begin{prop}
Suppose $f(X)=X^{3}-3a^{2}X+B$ is a cubic in $\mathbb{R}[X]$ with $B\geq0.$
\begin{enumerate}
\item Suppose that $0\leq a<1,$ one has that $\mathrm{J}(f)\nsubseteq\mathbb{R}$ for all $B\geq0.$
\item Let $A=-3a^{2}$ for $a\geq1$ and $B\geq0.$ Let $B_{0}>0$ be the unique positive real number with $\alpha_{1}(B_{0})=B_{0}-2a^{3},$ where $\alpha_{1}(B)$ is defined in Lemma \ref{fixed point decrease}. Then $\mathrm{J}(f)\subseteq\mathbb{R}$ for all $B\leq B_{0},$ and that $\mathrm{J}(f)\nsubseteq\mathbb{R}$ for all $B>B_{0}.$
\end{enumerate}
\end{prop}

\begin{proof}
Let $\alpha_{1}(B)$ and $\alpha_{2}(B)$ be the smallest and the largest fixed point of $f\in\mathbb{R}[X].$ Let $\beta_{1}(B)=B-2a^{3}$ and $\beta_{2}(B)=B+2a^{3}$ be the critical values of $f$ at its critical points $a$ and $-a$ respectively. Then we have $\mathrm{J}(f)\subseteq\mathbb{R}$ if and only if $\beta_{1}(B)\leq\alpha_{1}(B)$ and $\alpha_{2}(B)\leq\beta_{2}(B)$ because $[\alpha_{1},\alpha_{2}]\subseteq[\beta_{1},\beta_{2}].$

It is clear $\beta_{1}(B)=B-2a^{3}$ and $\beta_{2}(B)=B+2a^{3}$ are increasing functions on $B.$ Thus the fixed points $\alpha_{1}(B),\alpha_{2}(B)$ are decreasing functions on $B$ by Lemma \ref{fixed point decrease}, so it suffices to consider the cases when $\alpha_{1}(B)=\beta_{1}(B)$ and $\alpha_{2}(B)=\beta_{2}(B).$

Suppose that $0\leq a<1,$ then $X^{3}-(3a^{2}+1)X=0$ has roots $0$ and $\pm\sqrt{1+3a^{2}}.$ Now $(A,0)\in\mathcal{R}$ if and only if $\pm\sqrt{1+3a^{2}}\in[-2a^{3},2a^{3}].$ This holds if and only if 
\begin{equation}
\begin{split}
\sqrt{1+3a^{2}}\leq2a^{3}\Leftrightarrow&1+3a^{2}\leq4a^{6}\\\Leftrightarrow&4a^{6}-3a^{2}+1\geq0\\\Leftrightarrow&(a^{2}-1)(2a^{2}+1)^{2}\geq0.
\end{split}
\end{equation}
Now $2a^{2}+1$ is positive for all $a\in\mathbb{R},$ so the inequality holds if and only if $a^{2}-1\geq0,$ contradicting the assumption $0\leq a<1.$ As $\alpha_{1}(B)$ decreases and $\beta_{1}(B)$ increases, it follows that $\alpha_{1}(B)<\beta_{1}(B)$ for all $B$ and thus $(A,B)\notin\mathcal{R}$ for all $B\geq0.$

Suppose now that $a\geq1,$ then we have that $\beta_{1}(0)\leq\alpha_{1}(0)$ and $\alpha_{2}(0)\leq\beta_{2}(0)$ holds by an argument as above. It is always true that $\alpha_{2}(B)\leq\beta_{2}(B)$ for all $B\geq0$ as $\beta_{2}$ is increasing and $\alpha_{2}$ is decreasing. Since $\alpha_{1}(B)-\beta_{1}(B)$ is a decreasing and continuous function in $B$ and $\alpha_{1}(B)\rightarrow-\infty$ as $B\rightarrow\infty,$ it follows that there is $B_{0}\geq0$ maximum with $\beta_{1}(B_{0})=\alpha_{1}(B_{0}).$ The decreasing property of the map $\alpha_{1}(B_{0})-\beta_{1}(B_{0})$ implies $\beta_{1}(B)>\alpha_{1}(B)$ if $B>B_{0}$ and $\beta_{1}(B)<\alpha_{1}(B)$ if $B<B_{0}.$ This implies that $\mathrm{J}(f)\subseteq\mathbb{R}$ if $B\leq B_{0}$ and $\mathrm{J}(f)\nsubseteq\mathbb{R}$ if $B>B_{0}$ by Theorem \ref{Julia equivalence}. 
\end{proof}

The following result gives a condition when the Julia set of the polynomial $X^{3}+AX+B$ is contained in $\mathbb{R}.$

\begin{thm} \label{gray region}
Let $(A,B)\in\mathbb{R}^{2}$ and let $f_{A,B}=X^{3}+AX+B$ be its associated cubic. Then $\mathrm{J}(f_{A,B})\subseteq\mathbb{R}$ if and only if $A\leq-3$ and $B^{2}\leq-4A(A+3)^{2}/27.$
\end{thm}

\begin{proof}
Note $\mathrm{J}(f_{A,B})\subseteq\mathbb{R}$ if and only if $\mathrm{Fix}(f_{A,B},\mathbb{R})\subseteq\mathrm{I}(f_{A,B})$ by Proposition \ref{odd positive}. Set $A=-3a^{2}$ for some $a\geq0.$ So $f_{A,B}$ has critical values $B\mp2a^{2}$ at the points $X=\pm a,$ it follows that the critical interval is $[B-2a^{3},B+2a^{3}].$ Suppose that all fixed points are in the critical interval. The equation $X^{3}-3a^{2}X+B=B-2a^{3}$ factorizes $$X^{3}-3a^{2}X+2a^{3}=(X-a)^{2}(X+2a).$$ So $B-2a^{3}\leq -2a$ if the smallest fixed point is in the critical interval and thus $$B\leq 2a^{3}-2a=2a(a^{2}-1),$$ $$B^{2}\leq4a^{2}(a^{2}-1)^{2}=(-4/27)A(A+3)^{2}.$$
\end{proof}

\section{Lattès maps}
The main result in this section is an application of Theorem \ref{archimedean version} to establish a non-abelian result for the arboreal Galois group of certain Lattès maps. First we give an introduction to Lattès maps closely following Silverman \cite{MR2316407} Section $6.4.$

\begin{defn}
Let $C$ be a nonsingular curve defined over $\mathbb{C}.$ A \emph{uniformizer} at $P\in C_{1}$ is a rational function $t_{P}:C\rightarrow\mathbb{P}^{1}$ on $C$ that vanishes to order $1$ at $P.$
\end{defn}

\begin{defn}
Let $\varphi:C_{1}\rightarrow C_{2}$ be a nonconstant map between nonsingular curves $C_{1}$ and $C_{2}$ defined over $\mathbb{C}.$ The \emph{ramification index} $e_{\varphi}(P)$ at $P\in C_{1}$ is the multiplicity of $P$ as a solution to $\varphi(x)=Q$ as an equation in $x,$ where $Q=\varphi(P)$ is the image of $P$ under the nonconstant morphism $\varphi.$ To be precise $$e_{\varphi}(P)=\nu_{P}(t_{Q}\circ\varphi)$$ where $t_{P}$ is a uniformizer at $P,$ and $\nu_{P}$ is the order of the uniformizer $t_{P}\circ\varphi$ at $P.$ We say $\varphi$ is \emph{ramified} at $P$ if $e_{\varphi}(P)\geq2.$
\end{defn}

\begin{prop} \label{ramification multiplicative}
Let $\varphi:C_{1}\rightarrow C_{2}$ and $\psi: C_{2}\rightarrow C_{3}$ be nonconstant maps between three smooth curves $C_{1},C_{2},C_{3}.$
\begin{enumerate}
\item For all points $Q\in C_{2},$ it holds that $$\sum_{P\in\varphi^{-1}(Q)}e_{\varphi}(P)=\deg(\varphi).$$
\item For all points $P\in C_{1},$ the ramification index is multiplicative $$e_{\psi\circ\varphi}(P)=e_{\psi}(\varphi(P))e_{\varphi}(P).$$
\end{enumerate}
\end{prop}

\begin{proof}
This is proven in Silverman \cite{silverman:aec} Proposition $2.6.$
\end{proof}

\begin{prop} \label{Riemann Hurwitz}
If $E$ is an elliptic curve and if $\varphi:E\rightarrow E$ is an endomorphism, then $\varphi$ is unramified at all $P\in E.$
\end{prop}

\begin{proof}
It follows immediately from the Riemann-Hurwitz formula that $$\sum_{P\in E}(e_{\varphi}(P)-1)=0,$$ which is proven in Silverman \cite{MR2316407} Theorem $1.5,$ as $E$ is an elliptic curve of genus $1.$ Therefore the ramification index $e_{\varphi}(P)=1$ for all points $P\in E.$
\end{proof}

\begin{defn}
Let $E$ be an elliptic curve defined over $\mathbb{C}.$ The group of all complex \emph{$n$-torsion} points in $E(\mathbb{C})$ is denoted by $E[n].$ If in addition $E$ is defined over $\mathbb{R},$ the subgroup of all \emph{real $n$-torsion} points in $E(\mathbb{C})$ is denoted by $E[n](\mathbb{R}).$
\end{defn}

\begin{exam}
Let $E$ be an elliptic curve given by the Weierstrass equation $$y^{2}=F(x)=x^{3}+ax^{2}+bx+x.$$ Recall $x:E\rightarrow\mathbb{P}^{1}$ is the $x$-coordinate of the elliptic curve. If $P$ is a point on $E,$ then the classical formula for $x([2]P)$ yields the Lattès map $$f(X)=x([2]P)=\frac{X^{4}-2bX^{2}-8cX+b^{2}-4ac}{4(X^{3}+aX^{2}+bX+c)}=\frac{F^{\prime}(X)^{2}-(8X+4a)F(X)}{4F(X)},$$ satisfying $x\circ[2]=f\circ x$ such that the following diagram commutes 
\begin{equation}
\begin{tikzcd}
E\arrow[r,"{[2]}"]\arrow[d,"x"]&E\arrow[d,"x"]\\\mathbb{P}^{1}\arrow[r,"f"]&\mathbb{P}^{1},
\end{tikzcd}
\end{equation} proof is found in Silverman \cite{silverman:ataec} pages $59$ and $46,$ where $\psi(P)=[2]P$ is duplication, and $\pi$ is given by the $x$-coordinate function $\pi(P)=\pi(x,y)=x.$
\end{exam}

\begin{prop} \label{elliptic curve negative discriminant}
Let $E$ be an elliptic curve over $\mathbb{R},$ let $y^{2}=F(x)$ be a Weierstrass equation for $E,$ where $F\in\mathbb{R}[x]$ is a real cubic polynomial. Suppose that $\mathrm{disc}(F)<0.$ If $f:\mathbb{P}^{1}\rightarrow\mathbb{P}^{1}$ is a Lattès map associated to the duplication map $[2]:E\rightarrow E$ and $x:E\rightarrow\mathbb{P}^{1}$ with $x\circ[2]=f\circ x,$
then $f:\mathbb{P}^{1}(\mathbb{R})\rightarrow\mathbb{P}^{1}(\mathbb{R})$ is surjective as an endomorphism on real projective space.
\end{prop}

\begin{proof}
The duplication $[2]:E\rightarrow E$ is unramified at all $P\in E$ by Proposition \ref{Riemann Hurwitz}. As $x(-P)=x(P)$ for all $P,$ so $x$ ramifies at $P$ if and only if $-P=P$ which is $2P=\mathcal{O}.$ So $e_{x}(P)=2$ if $P$ is in $2$-torsion $E[2],$ and $e_{x}(P)=1$ if $P$ is not in $2$-torsion $E[2].$

Now $P$ is a critical point of $x\circ[2]$ if and only if $e_{x\circ[2]}(P)=e_{x}([2]P)e_{[2]}(P)>1.$ This happens if and only if $e_{x}([2]P)=2,$ as the map $[2]$ is unramified at all $P\in E.$ In other words, when $[2]P$ is a $2$-torsion point in $E[2],$ or equivalently when $P$ is a $4$-torsion point in $E[4].$

Consider a point $P$ at which $x\circ[2]$ ramifies, hence $P$ is a $4$-torsion point in $E[4].$ From the relation $x\circ[2]=f\circ x,$ the ramification index at $P$ can be written as $$e_{x\circ[2]}(P)=e_{f\circ x}(P)=e_{f}(x(P))e_{x}(P)=2.$$

Suppose that $e_{x}(P)=2,$ in other words $P$ is a $2$-torsion point on the curve $E.$ Then $e_{f}(x(P))=1$ and $x(P)$ is not a critical point of $f.$ Suppose that $e_{x}(P)=1,$ in other words, the point $P\in E$ is in $E[4]$ but $P$ is not in $E[2].$ Then we have that $e_{f}(x(P))=2$ and thus $x(P)$ is a critical point of $f.$ Thus we have shown that $x(P)$ is a critical point of $f$ if and only if $P$ is in $E[4]$ but is not in $E[2].$

As the discriminant of $F(x)$ is negative, it follows that $F$ has only one real root. Since $E(\mathbb{R})\cong\mathbb{R}/\mathbb{Z}$ as groups by Silverman in \cite{silverman:ataec} Corollary $\mathrm{V}.2.3.1,$ this implies $E[2](\mathbb{R})$ has order $2$ and $E[4](\mathbb{R})$ has order $4.$ It follows that $E[2](\mathbb{R})=\{\mathcal{O},T\},$ where $T=(\alpha,0)$ with $\alpha$ is the unique real root of $F(x),$ and the $4$-torsion subgroup $E[4](\mathbb{R})=\{\mathcal{O},T,\pm Q\}$ where $\pm Q$ have order $4$ as $E(\mathbb{R})\cong\mathbb{R}/\mathbb{Z}.$ There are four points $P\in E$ of order $4$ such that $[2]P=T,$ two of these are $\pm Q$ having real $y$-coordinate and the other two of these, say $\pm Q^{\prime}$ have nonreal $y$-coordinate.

If $\alpha$ is the only real root of $F,$ we have that $f(x)\rightarrow\pm\infty$ as $x\rightarrow\alpha^{\pm}$ and $x\rightarrow\pm\infty.$ This follows from the formula $$f(X)=\frac{F^{\prime}(X)^{2}-(8X+4a)F(X)}{4F(X)}$$ and that $\alpha$ is the unique real root of $F.$ So the denominator $F(x)>0$ for all $x>\alpha$ and also $F(x)<0$ for all $x<\alpha,$ and $F^{\prime}(x)^{2}>0$ implies that the numerator of $f(x)$ is positive if $x$ is close to $\alpha.$ Now $f(x)$ has critical points $c_{1}$ and $c_{2}$ with $c_{1}<\alpha<c_{2},$ where we note that $c_{1}=x(\pm Q^{\prime})$ arises from the two complex $4$-torsion points $\pm Q^{\prime},$ and that $c_{2}=x(\pm Q)$ arises from the two real $4$-torsion points $\pm Q$ on the curve $E.$ The fact that $f(c_{1})=f(c_{2})$ follows from the commutative relation $f\circ x=x\circ[2]$ as  $$f(c_{1})=f\circ x(Q^{\prime})=x\circ[2](Q^{\prime})=x\circ[2](Q)=f\circ x(Q)=f(c_{2}).$$ Since $f$ takes all values $y\leq f(c_{1})$ on $(-\infty,c_{1}],$ and all values $y\geq f(c_{2})$ on $[c_{2},+\infty),$ together with $f(c_{1})=f(c_{2}),$ it follows that $f$ is surjective on $\mathbb{P}^{1}(\mathbb{R}).$
\end{proof}

%The $x$-coordinate of the duplication map of $y^{2}=x^{3}+x-1$ is shown in Figure \ref{fig:same critical value}, the proof of the surjectivity of $f$ is obtained by showing that the two critical values are equal to each other.

\begin{exam}
The assumption $\mathrm{disc}(F)$ is negative in Proposition \ref{elliptic curve negative discriminant} is essential. Take the elliptic curve $$y^{2}=x^{3}-a^{2}x=x(x^{2}-a^{2})$$ for $a\ne0$ as an example. Thus the elliptic curve has four real $2$-torsion points and $$f(X)=\frac{(X^{2}+a^{2})^{2}}{4X(X^{2}-a^{2})}.$$ Since $f:\mathbb{P}^{1}(\mathbb{R})\rightarrow\mathbb{P}^{1}(\mathbb{R})$ never takes the value $0,$ its image omits a neighborhood of $0$ by compactness of $\mathbb{P}^{1}(\mathbb{R}).$ So $f$ is not surjective as an endomorphism on $\mathbb{P}^{1}(\mathbb{R}).$
\end{exam}

\begin{cor}
Let $K$ be a number field, and $E$ an elliptic curve defined over $K.$ Let $y^{2}=F(x)$ be a Weierstrass equation for $E,$ where $F\in\mathbb{R}[x]$ has degree three. Suppose that $\mathrm{disc}(F)<0.$  If $f:\mathbb{P}^{1}\rightarrow\mathbb{P}^{1}$ is a Lattès map associated to $[2]:E\rightarrow E$ and $x:E\rightarrow\mathbb{P}^{1}$ such that $x\circ[2]=f\circ x,$ and if $\alpha\in\mathbb{R}$ is a nonperiodic point of $f,$ then $\Gal(K_{\infty}/K)$ is not abelian.
\end{cor}

\begin{proof}
Since $f$ is surjective on $\mathbb{P}^{1}(\mathbb{R})$ by Proposition \ref{elliptic curve negative discriminant}, and the Julia set $\mathrm{J}(f)=\mathbb{P}^{1}(\mathbb{C})$ is not contained in $\mathbb{R}$ by Silverman \cite{MR2316407} Theorem $1.43,$ and $\alpha$ is not periodic, it follows that $\Gal(K_{\infty}/K)$ is not abelian by Theorem \ref{nonperiodic archimedean}.
\end{proof}

\bibliographystyle{plain}
\bibliography{CL}
\end{document}